\documentclass[11pt]{amsart}

\usepackage[T1]{fontenc}
\usepackage[utf8]{inputenc}
\usepackage{lmodern}
\usepackage{microtype}
\usepackage{amsmath,amssymb,amsthm,mathtools}
\usepackage[hidelinks]{hyperref}

\newcommand{\C}{\mathbb C}
\newcommand{\PP}{\mathbb P}
\newcommand{\Aff}{\mathbb A}
\newcommand{\gr}{\operatorname{gr}}
\newcommand{\rank}{\operatorname{rank}}
\newcommand{\Spec}{\operatorname{Spec}}
\newcommand{\Proj}{\operatorname{Proj}}
\newcommand{\TC}{\operatorname{TC}}

\newcommand{\QH}{\mathcal Q}

\newtheorem{theorem}{Theorem}[section]
\newtheorem*{maintheorem}{Theorem~\ref{thm:overview}}
\newtheorem{lemma}[theorem]{Lemma}
\newtheorem{proposition}[theorem]{Proposition}
\newtheorem{corollary}[theorem]{Corollary}
\theoremstyle{remark}

\title[Quadratic hulls of osculating varieties]
{Quadratic hulls of osculating varieties of rational normal curves}

\author{C\'esar Massri}
\address{Departamento de Matem\'atica, Facultad de Ciencias Exactas y Naturales, Universidad de Buenos Aires, Argentina}
\email{cmassri@dm.uba.ar}


\begin{document}

\begin{abstract}
Let $T^p c_r$ be the $p$-osculating variety of the rational normal curve
$c_r\subset\PP S^r(\C^2)$. We study its quadratic hull, the scheme defined
by the quadratic equations vanishing on $T^p c_r$. For $r\ge2p+3$, we prove
that this scheme is a complete intersection of $\binom{r-2p}{2}$ quadrics
whenever
\[
 p\ge\binom{r-2p-1}{2}.
\]
When the inequality is strict, the hull is integral and strictly contains
$T^p c_r$. On the boundary its support is reducible, except for
$(p,r)=(1,5)$, where the hull equals the osculating variety. The proof is local at the highest weight point. Koornwinder's relation
between Clebsch--Gordan coefficients and Hahn polynomials leads to a
description of the tangent-cone equations in which each antidiagonal becomes
an ordinary polynomial space of bounded degree. An incidence estimate then
gives a normal integral tangent cone in the strict range, and the $SL_2$-action
turns this local statement into global integrality.
\end{abstract}

\maketitle

\section{Introduction}

Throughout, we work over $\C$. Let $c_r\subset\PP S^r(\C^2)$ be the
rational normal curve and
\[
 T^p c_r
 =\{[\ell^{\,r-p}g]:0\ne\ell\in\C^2,\ 0\ne g\in S^p(\C^2)\}
\]
its $p$-osculating variety. Thus $T^p c_r$ is the locus of binary forms
with a root of multiplicity at least $r-p$. For a projective variety
$Y\subset\PP V$, write
\[
 \QH(Y)
 :=\Proj\!\left(
 \operatorname{Sym}(V^*)/\langle I(Y)_2\rangle
 \right)
\]
for the closed subscheme defined by the ideal generated by its quadratic
equations, and set
\[
 X_{p,r}=\QH(T^p c_r).
\]
The quadrics are the coefficients of the even transvectants $(f,f)_{2m}$,
$m>p$; equivalently,
\[
 I(T^p c_r)_2\simeq
 \bigoplus_{m=p+1}^{\lfloor r/2\rfloor}(S^{2r-4m}\C^2)^*;
\]
see \cite[Corollary~2.4]{Massri2026}. In the stable range $r>4p$ they
already cut out $T^p c_r$ set-theoretically
\cite[Corollary~2.6]{Massri2026}. Outside that range the common zero scheme
may be larger or reducible; the examples $(p,r)=(2,7)$ and $(3,10)$ show
both phenomena \cite[Proposition~4.1 and Introduction]{Massri2026}.

\begin{theorem}\label{thm:overview}
Let $p,r$ be integers with $p\ge0$ and $r\ge2p+3$, and set
\[
 N=\binom{r-2p-1}{2}.
\]
If $p\ge N$, then $X_{p,r}$ is a complete intersection of
$\binom{r-2p}{2}$ quadrics, with
\[
 \dim X_{p,r}=2p+1-N,
 \qquad
 \deg X_{p,r}=2^{\binom{r-2p}{2}}.
\]
If $p>N$, then $X_{p,r}$ is integral and
$T^p c_r\subsetneq X_{p,r}$. If $p=N$, its support is reducible except
when $(p,r)=(1,5)$; in that case $X_{1,5}=T^1c_5$ as schemes.
\end{theorem}

The numerical hypothesis is forced by dimension. The number of quadratic
equations is
\[
 M=\dim I(T^p c_r)_2=\binom{r-2p}{2},
\]
so their expected dimension is $r-M$, while $\dim T^p c_r=p+1$; the
difference is
\[
 (r-M)-(p+1)=p-N.
\]
Thus $p\ge N$ is necessary for these quadrics to form a complete
intersection. The proof that it is sufficient is local. Every irreducible
component of the support of $X_{p,r}$ is $SL_2$-stable, hence contains the
highest weight point $o=[x^r]$ by the Borel fixed point theorem. After eliminating the equations with independent linear parts at $o$, the
remaining quadratic initial forms split into antidiagonals. Koornwinder's
identification of the $SU(2)$ Clebsch--Gordan coefficients with Hahn
polynomials leads to a simple description of their coefficient functions:
after the pivot elimination, each antidiagonal becomes a full polynomial
space of bounded degree in the quadratic variable $i(L-i)$. Consequently,
making coefficients vanish costs parameters at the same rate at which the
rank of the corresponding quadratic form can drop. This gives a small singular
locus for the tangent cone, hence a normal integral tangent cone when
$p>N$. The associated graded local ring then gives a domain at $o$, and
the $SL_2$-action propagates this single local fact to the whole quadratic
hull. On the boundary, the same local analysis gives a complete
intersection, while a degree comparison with $T^p c_r$ forces
reducibility. The cases $p<N$ are not addressed by this argument.

As a by-product, Corollary~\ref{cor:individual-Mm} gives a range for the
individual summands considered in \cite{Massri2013}: the zero locus of
$S^{2r-4m}\C^2$ is a complete intersection if
$m-1\ge\binom{r-2m+1}{2}$, and is integral if the inequality is strict.

\section{Local equations at the highest weight point}\label{sec:local}

For fixed integers $p,r$ with $0\le p<r$, set
\[
 X_{p,r}=\QH(T^p c_r)\subset \PP S^r(\C^2).
\]
We regard $X_{p,r}$ as the scheme defined by the quadratic equations
$I(T^p c_r)_2$.

We use the same coordinates as in \cite{Massri2013}. Let
\[
 x_i=\binom ri x^{r-i}y^i,\qquad 0\le i\le r,
\]
so that a general binary form is written as
\[
 f=\sum_{i=0}^r a_i x_i.
\]
For every $0\le m\le\lfloor r/2\rfloor$, let
\[
 \pi_m:S^2(S^r\C^2)\longrightarrow S^{2r-4m}\C^2
\]
be a non-zero $SL_2$-equivariant projection. Choose the standard weight
basis $w_0^{(m)},\ldots,w_{2r-4m}^{(m)}$ and write
\begin{equation}\label{eq:q-expansion}
 \pi_m(f\cdot f)
 =\sum_{j=0}^{2r-4m}q_j^{(m)}(f)\,w_j^{(m)}.
\end{equation}
Since the summand $S^{2r-4m}\C^2$ occurs with multiplicity one, we may
normalize $\pi_m$ so that \eqref{eq:q-expansion} is the classical
$2m$-th transvectant $(f,f)_{2m}$. Weight considerations give
\begin{equation}\label{eq:qcoeff}
 q_j^{(m)}(f)
 =\sum_{i+i'=2m+j}
 q_j^{(m)}(x_i\cdot x_{i'})\,a_i a_{i'}.
\end{equation}
Thus each equation is supported on a single antidiagonal. With the present
normalization, for $i+i'=2m+j$ one has
\begin{equation}\label{eq:qexplicit}
 q_j^{(m)}(x_i\cdot x_{i'})
 =(-1)^i\sum_{\ell=0}^{j}(-1)^\ell
 \binom{r-2m}{\ell}
 \binom{r-2m}{j-\ell}
 \binom{2m}{i-\ell},
\end{equation}
where a binomial coefficient is zero if its lower index is outside the
usual range; compare \cite{Massri2013}.

Hahn polynomials are a classical and extensively studied family of discrete
orthogonal polynomials. Koornwinder's identification of the Clebsch--Gordan
coefficients for $SU(2)$ with Hahn polynomials \cite[\S4]{Koornwinder1981}
led us to recognize the following form of the coefficients
$q_j^{(m)}$. We use the normalization
\[
 Q_n(z;\alpha,\beta,N)
 =\sum_{k=0}^n
 \frac{(-n)_k\,(n+\alpha+\beta+1)_k\,(-z)_k}
 {(-N)_k\,(\alpha+1)_k\,k!},
\]
where $(u)_k$ denotes the Pochhammer symbol.

\begin{lemma}\label{lem:qpolynomial}
Let $L=2m+j\le r$ and put $a=r-L$. Then
\begin{equation}\label{eq:hahn}
 q_j^{(m)}(x_i\cdot x_{L-i})
 =(-1)^i\binom Li\binom{r-2m}{j}
 Q_j(i;a,a,L),
 \qquad 0\le i\le L.
\end{equation}
In particular, $Q_j(i;a,a,L)$ has degree $j$, satisfies
$Q_j(0;a,a,L)=1$, and
\begin{equation}\label{eq:hahn-symmetry}
 Q_j(L-i;a,a,L)=(-1)^jQ_j(i;a,a,L).
\end{equation}
If $L>r$, set
\[
 j'=2r-4m-j,
 \qquad
 L'=2m+j'=2r-L.
\]
Then
\begin{equation}\label{eq:reflection}
 q_j^{(m)}(x_i\cdot x_{L-i})
 =q_{j'}^{(m)}(x_{r-i}\cdot x_{L'-(r-i)}),
\end{equation}
and the first part applies to the right-hand side.
\end{lemma}

\begin{proof}
The maps $\pi_m$ are the multiplicity-one Clebsch--Gordan projections for
$S^r\C^2\otimes S^r\C^2$. Koornwinder's formula expresses their matrix
coefficients as Hahn polynomials. Translating from his orthonormal basis to
the binomial basis $x_i$ used here gives the factor
$(-1)^i\binom Li$ in \eqref{eq:hahn}; our normalization is then fixed by
\[
 q_j^{(m)}(x_0\cdot x_L)=\binom{r-2m}{j}.
\]
Equivalently, \eqref{eq:hahn} follows directly by rewriting the finite sum
\eqref{eq:qexplicit} as the terminating ${}_3F_2$ defining the Hahn
polynomial above. Formula \eqref{eq:hahn-symmetry} is the usual symmetry of
the equal-parameter Hahn polynomials; it also follows from the symmetry of
$x_i\cdot x_{L-i}$ in \eqref{eq:hahn}. Finally, the substitutions
$i\mapsto r-i$ and $\ell\mapsto r-2m-\ell$ in \eqref{eq:qexplicit} give
\eqref{eq:reflection}.
\end{proof}

By \cite[Corollary~2.4]{Massri2026}, the quadratic equations defining
$X_{p,r}$ are
\begin{equation}\label{eq:allq}
 q_j^{(m)}=0,
 \qquad
 p+1\le m\le\left\lfloor\frac r2\right\rfloor,
 \qquad
 0\le j\le 2r-4m.
\end{equation}
Their number is
\[
 \sum_{m=p+1}^{\lfloor r/2\rfloor}(2r-4m+1)
 =\binom{r-2p}{2}.
\]

We study these equations at $o=[x^r]$. Work in the chart $a_0=1$ and put
\[
 \mathfrak m=(a_1,\ldots,a_r).
\]
By \eqref{eq:qcoeff}, the only possible linear term in $q_j^{(m)}$ is
$a_0a_{2m+j}$. Thus
\begin{equation}\label{eq:qmodm2}
 q_j^{(m)}\equiv
 \begin{cases}
 2q_j^{(m)}(x_0\cdot x_{2m+j})\,a_{2m+j}
 & \pmod{\mathfrak m^2},\quad 2m+j\le r,\\[3pt]
 0 & \pmod{\mathfrak m^2},\quad 2m+j>r.
 \end{cases}
\end{equation}
Formula \eqref{eq:qexplicit} gives
\[
 q_j^{(m)}(x_0\cdot x_{2m+j})=\binom{r-2m}{j}\ne0
 \qquad (2m+j\le r).
\]

For $m=p+1$ and $0\le s\le r-2p-2$,
\begin{equation}\label{eq:pivotmodm2}
 q_s^{(p+1)}\equiv
 2\binom{r-2p-2}{s}a_{2p+2+s}
 \pmod{\mathfrak m^2}.
\end{equation}
These equations have independent initial linear forms
$a_{2p+2},\ldots,a_r$ and will be used as pivots.

Let $q_j^{(m)}$ be any other equation. If $2m+j\le r$, put
$s=2m+j-2p-2$ and define
\begin{equation}\label{eq:qhat}
 \widehat q_j^{(m)}
 :=q_j^{(m)}-
 \frac{\binom{r-2m}{j}}
 {\binom{r-2p-2}{s}}\,q_s^{(p+1)}.
\end{equation}
Its linear part vanishes. If $2m+j>r$, set simply
$\widehat q_j^{(m)}=q_j^{(m)}$.
Modulo the linear initial forms we may set
$a_{2p+2}=\cdots=a_r=0$; by abuse of notation we regard the resulting
quadratic initial forms $\widehat q_j^{(m)}$ as elements of
$\C[a_1,\ldots,a_{2p+1}]_2$. Every monomial left has index sum at least
$2p+4$, so both indices are at least $3$. Hence these forms involve only
$a_3,\ldots,a_{2p+1}$. There are
\[
 N=\binom{r-2p-1}{2}
\]
of them. We write
\[
 W=\operatorname{span}\{\widehat q_j^{(m)}\}
 \subset \C[a_3,\ldots,a_{2p+1}]_2.
\]

\section{The initial quadratic forms}\label{sec:W}

For the remainder of Sections~\ref{sec:W}--\ref{sec:normal}, assume
\[
 r\ge2p+3,
 \qquad
 p\ge\binom{r-2p-1}{2}.
\]
The point of the space $W$ is that the original transvectants have now
collapsed to a simpler object. By Lemma~\ref{lem:qpolynomial}, the
coefficient functions on a fixed antidiagonal are Hahn polynomials with the
same parameters and alternating degrees. The pivot subtraction has a simple
effect on this family, and each antidiagonal becomes an ordinary polynomial
space in the quadratic variable $i(L-i)$.

For $0\le s\le 2r-4p-6$, set
\[
 L_s=2p+4+s
\]
and let $W_s\subset W$ be the span of the forms $\widehat q_j^{(m)}$ with
$2m+j=L_s$. Different $W_s$ have disjoint monomial support, so
\[
 W=\bigoplus_s W_s.
\]
Set
\[
 \mu_s=
 \left\lfloor
 \frac{\min(s,\,2r-4p-6-s)}{2}
 \right\rfloor.
\]
Up to the middle antidiagonal, the sequence is
$0,0,1,1,2,2,\ldots$, and then it is read backwards.

\begin{lemma}\label{lem:Wstructure}
For every $s$ there is a linear isomorphism
\[
 \Phi_s:W_s\longrightarrow \C[z]_{\le\mu_s}
\]
such that, if $P=\Phi_s(q)$, then for
$3+s\le i\le2p+1$,
\begin{equation}\label{eq:coeffzero}
 q(x_i\cdot x_{L_s-i})=0
 \quad\Longleftrightarrow\quad
 P\bigl(i(L_s-i)\bigr)=0.
\end{equation}
Consequently,
\[
 \dim W_s=\mu_s+1,
 \qquad
 \dim W=\binom{r-2p-1}{2}=N.
\]
Moreover, prescribing the values of $P$ at
$b\le\mu_s+1$ distinct points imposes $b$ independent linear conditions.
\end{lemma}

\begin{proof}
Write $L=L_s$. Suppose first that $L\le r$, and put $a=r-L$. The pivot on
this antidiagonal is $q_{s+2}^{(p+1)}$. The remaining equations have Hahn
degrees
\[
 j=s,s-2,\ldots,\varepsilon,
 \qquad \varepsilon\equiv s\pmod2,
 \quad \varepsilon\in\{0,1\}.
\]
Divide the coefficient values by $(-1)^i\binom Li$. By
Lemma~\ref{lem:qpolynomial} and the choice of the scalar in
\eqref{eq:qhat}, the forms spanning $W_s$ then correspond, up to rescaling
the individual generators, to
\[
 Q_j(i;a,a,L)-Q_{s+2}(i;a,a,L),
 \qquad j=s,s-2,\ldots,\varepsilon.
\]
The Hahn polynomials
\[
 Q_\varepsilon,Q_{\varepsilon+2},\ldots,Q_{s+2}
\]
have distinct degrees and the same symmetry under $i\mapsto L-i$, hence
form a basis of the polynomials of degree at most $s+2$ with that symmetry.
Since all of them take the value $1$ at $i=0$, the displayed differences
form a basis of the kernel of evaluation at $0$ in this space.

If $s$ is even, every polynomial in this kernel is symmetric, vanishes at
$i=0$ and $i=L$, and therefore has a unique expression
\[
 i(L-i)P\bigl(i(L-i)\bigr),
 \qquad \deg P\le s/2=\mu_s.
\]
If $s$ is odd, it is antisymmetric and has the unique form
\[
 i(L-i)(L-2i)P\bigl(i(L-i)\bigr),
 \qquad \deg P\le (s-1)/2=\mu_s.
\]
On the active indices $3+s\le i\le2p+1$, the factors $i(L-i)$ do not
vanish; in the odd case $L$ is odd, so $L-2i$ does not vanish either.
This proves \eqref{eq:coeffzero} and gives the required isomorphism when
$L\le r$.

Suppose now that $L>r$. Put
\[
 L'=2r-L,
 \qquad
 t=2r-4p-6-s,
 \qquad
 k=r-i.
\]
By \eqref{eq:reflection}, after reflection the coefficient functions are the
equal-parameter Hahn polynomials of degrees
\[
 t,t-2,\ldots,\varepsilon',
 \qquad \varepsilon'\equiv t\pmod2.
\]
They therefore form the full space of polynomials of degree at most $t$
with the corresponding symmetry under $k\mapsto L'-k$. If $t$ is even,
this space consists of the polynomials in
$u=k(L'-k)$ of degree at most $t/2$; if $t$ is odd, it consists of
$(L'-2k)$ times such polynomials of degree at most $(t-1)/2$. On the active
indices the latter factor never vanishes. Since
\[
 k(L'-k)=i(L-i)+r(r-L),
\]
translation of the variable identifies this space with
$\C[z]_{\le\mu_s}$, where $z=i(L-i)$. This again gives
\eqref{eq:coeffzero}.

Thus $\dim W_s=\mu_s+1$ for every $s$. Summing over the antidiagonals gives
\[
 \dim W=\sum_s(\mu_s+1)=\binom{r-2p-1}{2}=N.
\]
The last assertion is ordinary polynomial interpolation.
\end{proof}

\section{The incidence estimate}\label{sec:rank}

Let
\[
 N=\dim W=\binom{r-2p-1}{2},
 \qquad
 n=2p-1,
\]
and choose a basis $q_1,\ldots,q_N$ of $W$. Put
\[
 \Sigma=
 \{a\in\Aff^n:\nabla q_1(a),\ldots,\nabla q_N(a)
 \text{ are linearly dependent}\}.
\]
For a quadratic form $q$, let $M_q$ be its symmetric matrix. Since
$\nabla q(a)=2M_q a$, define
\[
 \mathcal I=
 \{([q],a)\in\PP W\times\Aff^n:M_q a=0\}.
\]
Its projection to $\Aff^n$ is $\Sigma$. The estimate below is the only
rank calculation needed later.

\begin{lemma}\label{lem:incidence}
With the notation above,
\[
 \dim\mathcal I\le N-1.
\]
Consequently, $\dim\Sigma\le N-1$.
\end{lemma}

\begin{proof}
For each $s$, put
\[
 D_s=\sum_{\ell<s}(\mu_\ell+1).
\]
Let $q\ne0$ and let $s$ be the first index for which its $W_s$-component
is non-zero. The vanishing of the preceding components costs $D_s$
parameters, so such $[q]$ form a subset of $\PP W$ of dimension
$N-D_s-1$.

Let $P$ be the coefficient polynomial of the first non-zero
antidiagonal. Suppose it has $b$ distinct roots among the active values
$i(L_s-i)$. By Lemma~\ref{lem:Wstructure}, these are $b$ independent
conditions, so the corresponding forms have dimension at most
\[
 N-D_s-1-b,
\]
and $b\le \mu_s$.

The first non-zero antidiagonal occupies the square block indexed by
$3+s,\ldots,2p+1$, of size $n-s$. Since $s$ is minimal, all earlier
antidiagonals vanish, while the later ones lie strictly on one side of this
antidiagonal. Thus the $W_s$-antidiagonal is the leading antidiagonal of a
triangular minor, and its determinant is, up to sign, the product of its
entries. If $e$ entries on this antidiagonal vanish, omit, for each of them, its row
and the column paired with it by the antidiagonal. The remaining
$(n-s-e)\times(n-s-e)$ minor has a non-zero leading antidiagonal and hence
non-zero determinant.
Hence
\[
 \rank M_q\ge n-s-e.
\]
Each root of $P$ accounts for at most a symmetric pair of zero
coefficients, so $e\le2b$. Therefore
\[
 \dim\ker M_q\le s+2b.
\]
The corresponding part of $\mathcal I$ has dimension at most
\begin{equation}\label{eq:incidence-count}
 (N-D_s-1-b)+(s+2b)
 =N-1+s+b-D_s.
\end{equation}

It remains to compare $D_s$ with $s+\mu_s$. Set
$E_s=D_s-s-\mu_s$. Since $D_{s+1}=D_s+\mu_s+1$,
\[
 E_{s+1}=E_s+2\mu_s-\mu_{s+1}.
\]
Now $E_0=E_1=0$. The only possible drop is from $s=1$ to $s=2$; when
$\mu_2=1$, one has $E_2=-1$. Then $\mu_3\le1$, so $E_3\ge0$. For $s\ge3$ the
sequence $\mu_s$ changes by at most one and $2\mu_s-\mu_{s+1}\ge0$. Thus
\begin{equation}\label{eq:Ds-bound}
 D_s\ge s+\mu_s
\end{equation}
for every $s\ne2$. Since $b\le \mu_s$, \eqref{eq:incidence-count} is at most
$N-1$ in all these cases.

Consider the exceptional case $s=2$, where $\mu_2=1$ and $D_2=2$.
If $b=0$, the parameter space has dimension $N-3$ and the leading
antidiagonal gives corank at most $2$, so the incidence has dimension at
most $N-1$. If $b=1$ and the zero is central, the parameter space has
dimension $N-4$ and the corank is at most $3$, with the same conclusion.
It remains to consider the case in which the unique root of $P$ gives two
symmetric non-central zeros. The parameter stratum has dimension $N-4$,
and the leading antidiagonal alone gives only
$\operatorname{corank}M_q\le4$.

We claim that the generic corank on this stratum is at most $3$. By
Lemma~\ref{lem:Wstructure}, there exists $q_3\in W_3$ with
$\Phi_3(q_3)=1$; hence none of the entries on its antidiagonal vanishes.
For a form whose $W_2$-component lies in the stratum above, whose
$W_3$-component is $q_3$, and whose higher components vanish, the two
successive antidiagonals form a bidiagonal block: the leading one has
precisely the prescribed pair of zeros, while the next one has none. This
block has a non-zero minor of size $n-3$. Therefore some
$(n-3)\times(n-3)$ minor $\Delta$ of $M_q$ is not identically zero on the
parameter stratum.

On the open subset $\Delta\ne0$, the corank is at most $3$, and the
corresponding part of $\mathcal I$ has dimension at most
\[
 (N-4)+3=N-1.
\]
On the closed subset $\Delta=0$, the parameter dimension is at most
$N-5$; the original bound $\operatorname{corank}M_q\le4$ then gives again
\[
 (N-5)+4=N-1.
\]
Thus
$\dim\mathcal I\le N-1$.
\end{proof}

\section{The tangent cone}\label{sec:normal}

Let $J=\langle W\rangle$ in $\C[a_3,\ldots,a_{2p+1}]$ and set
\[
 C=\Spec\bigl(\C[a_3,\ldots,a_{2p+1}]/J\bigr).
\]
We refer to $C$ simply as the cone. The two variables $a_1,a_2$ do not
occur in $W$. The tangent cone of $X_{p,r}$ at $o$ is
\[
 \TC_o(X_{p,r})
 :=\Spec\!\left(
 \gr_{\mathfrak m}\mathcal O_{X_{p,r},o}
 \right),
\]
where $\mathfrak m$ is the maximal ideal of $o$. Once the initial forms
are known to be a regular sequence, we shall obtain
$\TC_o(X_{p,r})\simeq\Aff^2\times C$. Thus all that remains locally is
to understand the cone $C$. By Lemma~\ref{lem:incidence}, the locus where
the gradients of the forms in $W$ become dependent has dimension at most
$N-1$.

\begin{proposition}\label{prop:normal}
Assume
\[
 p\ge N=\binom{r-2p-1}{2}.
\]
Then $J$ is a complete intersection of height $N$. If $p>N$, then
\[
 \C[a_3,\ldots,a_{2p+1}]/J
\]
is a normal domain.
\end{proposition}

\begin{proof}
Recall that $n=2p-1$. Suppose a component of the support of $C$ had
codimension $c<N$. Its dimension would be at least
\[
 n-N+1=2p-N\ge N.
\]
At every smooth point of that component, the gradients of the generators
of $J$ have rank at most $c$, so the component would lie in $\Sigma$.
This contradicts $\dim\Sigma\le N-1$. Thus every minimal prime over $J$
has height at least $N$. Since $J$ has $N$ generators, Krull's height
theorem gives $\operatorname{ht}J=N$. Hence the generators form a regular
sequence.

Assume now $p>N$. Then
\[
 \dim C=n-N=2p-1-N.
\]
Outside $\Sigma$, the Jacobian criterion gives regularity. Hence
\[
 \operatorname{codim}_C \operatorname{Sing}(C)
 \ge(n-N)-(N-1)=2(p-N)\ge2.
\]
The coordinate ring of $C$ is a complete intersection, hence satisfies
$S_2$. Serre's criterion therefore gives normality; equivalently, one may
apply \cite[Chap.~II, Prop.~8.23]{Hartshorne}.

Every irreducible component of the affine cone $C$ contains the origin. A
normal scheme cannot have two distinct irreducible components meeting at a
point, since the local ring there would not be a domain. Therefore $C$ has
one irreducible component, and its coordinate ring is a normal domain.
\end{proof}

\section{From the tangent cone to the quadratic hull}\label{sec:global}

We use the following standard filtered-ring fact. If $(A,\mathfrak m)$ is
a Noetherian local ring and $g_1,\ldots,g_s\in A$ have initial forms that
form a regular sequence in $\gr_{\mathfrak m}A$, then
\begin{equation}\label{eq:VV}
 \gr_{\mathfrak m}\bigl(A/(g_1,\ldots,g_s)\bigr)
 \simeq
 \frac{\gr_{\mathfrak m}A}
 {(\operatorname{in}g_1,\ldots,\operatorname{in}g_s)}.
\end{equation}
This follows, for example, from \cite[Theorem~2.3]{ValabregaValla}.

\begin{proposition}\label{prop:globalCI}
Assume
\[
 r\ge2p+3,
 \qquad
 p\ge N=\binom{r-2p-1}{2}.
\]
Then
\begin{equation}\label{eq:local-graded}
 \gr_{\mathfrak m}\mathcal O_{X_{p,r},o}
 \simeq
 \C[a_1,a_2]\otimes_{\C}
 \frac{\C[a_3,\ldots,a_{2p+1}]}{J}.
\end{equation}
Equivalently,
\[
 \TC_o(X_{p,r})\simeq\Aff^2\times C.
\]
Moreover, $X_{p,r}\subset\PP^r$ is a complete intersection of
\[
 M=\binom{r-2p}{2}
\]
quadrics. In particular,
\[
 \dim X_{p,r}=r-M=2p+1-N,
 \qquad
 \deg X_{p,r}=2^M.
\]
\end{proposition}

\begin{proof}
After the elimination in Section~\ref{sec:local}, the local equations at
$o$ consist of $N$ equations in $a_1,\ldots,a_{2p+1}$ whose quadratic
initial forms span $W$. Proposition~\ref{prop:normal} says that these forms
are a regular sequence. Formula \eqref{eq:VV} therefore gives
\eqref{eq:local-graded}, and
\[
 \dim\mathcal O_{X_{p,r},o}=2p+1-N.
\]

Every irreducible component of the support of $X_{p,r}$ is $SL_2$-stable,
because $SL_2$ is connected. The unique fixed point of the upper
triangular Borel subgroup in $\PP S^r\C^2$ is $o$. By the Borel fixed point
theorem, every component contains $o$. Hence
\[
 \dim X_{p,r}=2p+1-N.
\]
The homogeneous ideal of $X_{p,r}$ is generated by the
$M=\binom{r-2p}{2}$ quadrics in \eqref{eq:allq}, and
\[
 \operatorname{codim}_{\PP^r}X_{p,r}
 =r-(2p+1-N)
 =\binom{r-2p}{2}=M.
\]
Thus these quadrics form a regular sequence. The dimension and degree
formulas follow.
\end{proof}

We restate the main theorem before completing its proof.

\begin{maintheorem}
Let $p,r$ be integers with $p\ge0$ and $r\ge2p+3$, and set
$N=\binom{r-2p-1}{2}$. If $p\ge N$, then $X_{p,r}$ is a complete
intersection of $\binom{r-2p}{2}$ quadrics, with
\[
 \dim X_{p,r}=2p+1-N,
 \qquad
 \deg X_{p,r}=2^{\binom{r-2p}{2}}.
\]
If $p>N$, then $X_{p,r}$ is integral and
$T^p c_r\subsetneq X_{p,r}$. If $p=N$, its support is reducible except
when $(p,r)=(1,5)$; in that case $X_{1,5}=T^1c_5$ as schemes.
\end{maintheorem}

\begin{proof}
Proposition~\ref{prop:globalCI} gives the complete-intersection statement
and the formulas for dimension and degree. Assume first that $p>N$.
By Proposition~\ref{prop:normal}, the right-hand side of
\eqref{eq:local-graded} is a domain. Since the maximal-ideal filtration is
separated, $\mathcal O_{X_{p,r},o}$ is a domain.

Every irreducible component of the support contains $o$, so the support has
only one component. A complete intersection is Cohen--Macaulay and has no
embedded associated primes. The local ring at $o$ is reduced, hence the
unique component is generically reduced. Therefore $X_{p,r}$ is integral.
Moreover,
\[
 \dim X_{p,r}-(p+1)=p-N>0,
\]
so the inclusion $T^p c_r\subseteq X_{p,r}$ is strict.

Now assume $p=N$ and $r\ge2p+4$. Then $X_{p,r}$ has dimension $p+1$ and
degree $2^M$, where $M=\binom{r-2p}{2}$. The osculating variety is
irreducible and, by the Hilbert polynomial computed in
\cite{Weyman1989},
\[
 \dim T^p c_r=p+1,
 \qquad
 \deg T^p c_r=(p+1)(r-p).
\]
Set $t=r-2p-1$. The boundary equality gives $p=\binom t2$ and $t\ge3$,
while $r-p=p+t+1$. Hence
\[
 p+1<r-p<2(p+1).
\]
The two factors $p+1$ and $r-p$ cannot both be powers of two, since two
distinct powers of two have ratio at least two. Thus
$\deg T^p c_r$ has an odd prime divisor and cannot divide $2^M$.

If the support of $X_{p,r}$ were irreducible, then, because it contains the
irreducible variety $T^p c_r$ of the same dimension, the two supports would
coincide. The degree of the complete intersection would then be a positive
integer multiple of $\deg T^p c_r$, a contradiction. Hence the support is
reducible. This degree argument is the same one used for the case
$(p,r)=(3,10)$ in \cite[Introduction]{Massri2026}.

It remains to consider $(p,r)=(1,5)$. This lies in the stable range of
\cite[Corollary~2.6]{Massri2026}, so the support of $X_{1,5}$ is $T^1c_5$.
Proposition~\ref{prop:globalCI} gives $\deg X_{1,5}=8$, while
\cite{Weyman1989} gives $\deg T^1c_5=8$. Hence the unique component occurs
with multiplicity one. Since a complete intersection is Cohen--Macaulay
and has no embedded associated primes, $X_{1,5}$ is reduced. Therefore
$X_{1,5}=T^1c_5$ as schemes.
\end{proof}

We finish with a consequence for the individual irreducible summands of the
quadratic equations. Recall the varieties considered in \cite{Massri2013}:
for $0\le m\le\lfloor r/2\rfloor$, let
\[
 M_m=\{[f]\in\PP S^r(\C^2):\pi_m(f\cdot f)=0\}.
\]
Thus $M_m$ is defined by the $2r-4m+1$ quadrics corresponding to the
summand $S^{2r-4m}\C^2$.

\begin{corollary}\label{cor:individual-Mm}
Let $r\ge2m+1$. If
\[
 m-1\ge\binom{r-2m+1}{2},
\]
then $M_m$ is a complete intersection of $2r-4m+1$ quadrics. In particular,
\[
 \dim M_m=4m-r-1,
 \qquad
 \deg M_m=2^{\,2r-4m+1}.
\]
If the inequality is strict, then $M_m$ is integral.
\end{corollary}

\begin{proof}
Take $p=m-1$ in Proposition~\ref{prop:globalCI}. Under the hypothesis of the
corollary, the quadrics defining $X_{m-1,r}$ form a regular sequence. The
equations of $M_m$ are a subfamily of this sequence, hence form a regular
sequence as well. This proves the complete-intersection statement and the
formulas for dimension and degree.

Assume that the inequality is strict and set
\[
 A=\mathcal O_{M_m,o}.
\]
Let $I\subset A$ be generated by the remaining equations of $X_{m-1,r}$.
They form an $A$-regular sequence of length $\binom{r-2m}{2}$, and
\[
 A/I=\mathcal O_{X_{m-1,r},o}.
\]
The latter is a domain by Theorem~\ref{thm:overview}; hence
\[
 \operatorname{gr}_I A
 \simeq
 \mathcal O_{X_{m-1,r},o}
 [T_1,\ldots,T_{\binom{r-2m}{2}}]
\]
is a domain. Since the $I$-adic filtration is separated, $A$ is a domain.

Every irreducible component of $M_m$ is $SL_2$-stable and therefore contains
$o$ by the Borel fixed point theorem. Thus $M_m$ has only one component.
Since it is a complete intersection, it is Cohen--Macaulay and has no embedded
associated primes; being reduced at $o$, its unique component is generically
reduced. Hence $M_m$ is integral.
\end{proof}

When $r=2m$, the variety $M_m$ is the nondegenerate invariant quadric
described in \cite[Example~11]{Massri2013}, and is therefore integral as
well.

\end{document}